\documentclass{amsart}
\usepackage{amsmath}
\usepackage{mathtools}

\usepackage{graphicx}
\usepackage[colorlinks=true, allcolors=blue]{hyperref}
\usepackage{amsfonts}
\usepackage{amsthm}
\usepackage{newlfont}
\usepackage{amscd}
\usepackage{amsgen}
\usepackage{amssymb} 
\usepackage{mathrsfs}	
\usepackage{longtable}
\usepackage{listings}
\usepackage{extarrows}
\usepackage{tikz}
\usepackage{tikz-cd}
\usepackage{verbatim}
\numberwithin{equation}{section}
\usepackage[all]{xy}
\usepackage{color}
\usepackage{amssymb}
\usepackage[left=3cm,right=3cm]{geometry}
\usepackage{tikz}
\usepackage{tikz-cd}
\usepackage{mathtools}
\usepackage{bm} 
\usepackage{dynkin-diagrams}
\usetikzlibrary{matrix,shapes,arrows,decorations.pathmorphing}
\usepackage{calligra}
\usepackage{mathrsfs}
\usepackage{enumitem} 
\usepackage{tgpagella} 
\usepackage[parfill]{parskip} 
\usepackage{float}
\restylefloat{table}
\usepackage{ytableau}
\usepackage{adjustbox}
\usetikzlibrary{calc,matrix,positioning}

\tikzset{ 
    table/.style={
        matrix of nodes,
        row sep=-\pgflinewidth,
        column sep=-\pgflinewidth,
        nodes={rectangle, text width=2.5em, text height = 1.5em, align=center},
        text depth=1.25ex,
        text height=2.5ex,
        nodes in empty cells
    },
}

\let\blb\mathbb

\def\CC{{\blb C}}

\def\LL{{\blb L}}

\def\PP{{\blb P}}
\def\QQ{{\blb Q}}
\def\RR{{\blb R}}

\def\VV{{\blb V}}

\def\ZZ{{\blb Z}}

\let\cal\mathcal
\def\Ac{{\cal A}}
\def\Bc{{\cal B}}
\def\Cc{{\cal C}}

\def\Ec{{\cal E}}
\def\Fc{{\cal F}}

\def\Kc{{\cal K}}

\def\Nc{{\cal N}}
\def\Oc{{\cal O}}

\def\Qc{{\cal Q}}
\def\Rc{{\cal R}}
\def\Sc{{\cal S}}
\def\Tc{{\cal T}}
\def\Uc{{\cal U}}

\def\Wc{{\cal W}}

\newtheorem{lemma}{Lemma}[section]
\newtheorem{proposition}[lemma]{Proposition}
\newtheorem{theorem}[lemma]{Theorem}
\newtheorem{corollary}[lemma]{Corollary}

\newtheorem{definition}[lemma]{Definition}

\newtheorem{todo}[lemma]{To do}

\theoremstyle{remark}

\newtheorem{remark}[lemma]{Remark}

\def\dbcoh{D^b}
\def\coh{\operatorname{coh}}
\def\wt{\widetilde}
\def\wh{\widehat}
\def\arw{\longrightarrow}
\def\Hom{\operatorname{Hom}}
\def\coker{\operatorname{coker}}

\def\im{\operatorname{im}}
\def\deg{\operatorname{deg}}
\def\Ext{\operatorname{Ext}}

\def\cone{\operatorname{Cone}}

\def\Span{\operatorname{Span}}

\DeclareMathOperator{\Sym}{Sym}

\newcommand{\forgetful}{\operatorname{Fg}}

\DeclareMathOperator{\sHom}{\mathscr{H}\text{\kern -3pt {\calligra\large om}}\,}

\newcommand\quotient[2]{
        \mathchoice
            {
                \text{\raise1ex\hbox{$#1$}\Big/\lower1ex\hbox{$#2$}}%
            }
            {
                #1\,/\,#2
            }
            {
                #1\,/\,#2
            }
            {
                #1\,/\,#2
            }
    }

\makeatletter
\def\namedlabel#1#2{\begingroup
    #2%
    \def\@currentlabel{#2}%
    \phantomsection\label{#1}\endgroup
}
\makeatother

\title[Fullness of the KP exceptional collection for the spinor tenfold]{Fullness of the Kuznetsov--Polishchuk exceptional\\ collection for the spinor tenfold}

\author[R. Moschetti]{Riccardo Moschetti}
\address{RM: Department of Mathematics G. Peano, University of Turin, via Carlo Alberto 10, 10123 Torino, Italy} 
\email{riccardo.moschetti@unito.it}

\author[M. Rampazzo]{Marco Rampazzo}
\address{MR: 
Department of Mathematics \\ University of Bologna \\ Piazza di Porta San Donato 5\\ 40126 Bologna, Italy.}
\email{marco.rampazzo3@unibo.it, marco.rampazzo.90@gmail.com}

\begin{document}

\maketitle

\begin{abstract}
    Kuznetsov and Polishchuk provided a general algorithm to construct exceptional collections of maximal length for homogeneous varieties of type $A,B,C,D$. We consider the case of the spinor tenfold and we prove that the corresponding collection is full, i.e. it generates the whole derived category of coherent sheaves. As a step of the proof, we construct some resolutions of homogeneous vector bundles which might be of independent interest.
\end{abstract}

\section{Introduction}
The derived category of coherent sheaves has been a subject of study in algebraic geometry for several decades: proven to be a powerful invariant for Fano and general type varieties by means of the reconstruction theorem \cite{bondalorlovreconstruction} and for $K3$ surfaces due to the derived Torelli theorem \cite{orlov_derived_torelli, huybrechts_torelli}, in general the interplay of the derived category with other invariants -- such as the Hodge structure and the class in the Grothendieck ring of varieties -- has yet to be fully understood. 
While these categories are usually complicated to describe, a common approach to understanding their structure is by providing a so-called \emph{full exceptional collection}, which is, roughly speaking, a finite list of objects which generate the whole category and satisfy simple cohomological conditions. While it is well-known that some classes of varieties cannot admit full exceptional collection, it is very much expected that rational homogeneous varieties do.\\
\\
When it comes to the problem of constructing full exceptional collections for rational homogeneous varieties, the techniques developed in \cite{beilinson} for projective spaces have been successfully generalized to Grassmannians and smooth quadrics in \cite{kapranov_Grassmannians}.
However, a general approach is still missing: only partial answers have been given in the case of classical and exceptional Grassmannians, mostly using specific, ad-hoc techniques. 
Up to products, every rational homogeneous variety can be obtained as a quotient $G/P^i$ where $P^i\subset G$ is the maximal parabolic subgroup corresponding to the $i$-th simple root (more on this notation in Section \ref{sec:vector_bundles}). Nowadays, to the best of the authors' knowledge, full exceptional collections have been found for the following classes for every positive integer $n$, denoting the number of nodes in the corresponding Dynkin diagram:
\begin{itemize}
    \item $G = A_n$, $P = P^k$ for any $k$ \cite{beilinson, kapranov_Grassmannians},
    \item $G = B_n$, $P = P^1$ (\cite{kapranov_Grassmannians}) or $P= P^2$ (\cite{KuznetsovExcCollGrassIsoLines}),
    \item $G = C_n$, $P = P^1$ (\cite{beilinson, kapranov_Grassmannians}), $P= P^2$  (\cite{KuznetsovExcCollGrassIsoLines}) or $P= P^n$ (\cite{FonarevLagrangian}),
    \item $G = D_n$, $P = P^1$ (\cite{kapranov_Grassmannians}) or $P= P^2$ (\cite{kuznetsovSmirnovResidual}),
    
\end{itemize}
Moreover, the following sporadic cases are also known to admit a full exceptional collection:
\begin{itemize}
    \item $G = B_3$, $P = P^3$ (\cite{kapranov_Grassmannians}) and $G = B_4$, $P = P^4$ (\cite{KuznetsovHyperplane}, \cite{KuznetsovExcCollGrassIsoLines}),
    \item $G = C_4$, $P = P^3$ (\cite{GusevaIGr}) and $G = C_5$, $P = P^3$ (\cite{NovikovThesis}),
    \item $G = D_i$, $P = P^i$ or $P = P^{i-1}$  for $i = 4$ (\cite{kapranov_Grassmannians}) and $i=5$ (\cite{KuznetsovHyperplane}, \cite{KuznetsovExcCollGrassIsoLines}),
    \item $G=E_6$, $P= P^1$ or $P= P^6$ (\cite{FaenziManivelCayleyII}),
    \item $G=F_4$, $P= P^1$ (\cite{SmirnovTypeFArxiv}) or $P= P^4$ (\cite{BKSCayley}),
    \item $G = G_2$, $P = P^1$ (\cite{kapranov_Grassmannians}) or $P= P^2$ (\cite{KuznetsovHyperplane}).
\end{itemize}
A crucial step towards a general answer to this problem is provided by Kuznetsov and Polishchuk in  \cite{KP_collections_isotropic}: there, exceptional collections of maximal length are given for all Grassmannians of type $A, B, C, D$, where the maximal length of an exceptional collection for $\dbcoh(X)$ with $X$ smooth projective is $\sum_p h^{(p,p)}(X)$ (see \cite[Corollary 2.16]{kuznetsov_derived_category_view}). In the following, we will refer to them as KP collections. 
However, fullness is still conjectural: in other words, it is not known whether the KP collection for a given $G/P$ generates the whole $\dbcoh(G/P)$. In this paper, we perform the construction of Kuznetsov and Polishchuk explicitly for the spinor tenfold (i.e. a connected component of the Grassmannian of isotropic five spaces in a ten dimensional vector space) obtaining:
\begin{equation}
    \label{eq:KP_collection_X}
    \begin{split}
        \langle &\Oc, \Oc(1), \Oc(2), \Uc^\vee(2), \Sym^2\Uc^\vee(2), \Oc(3), \Uc^\vee(3), \Sym^2\Uc^\vee(3), \\
        & \hspace{40pt}\Oc(4), \Uc^\vee(4),    \wh T_{ X }(5), \Oc(5), \Uc^\vee(5),  \wh T_{ X }(6), \Oc(6), \Oc(7) \rangle
    \end{split}
\end{equation}
where $\Uc^\vee$ is the dual of the pullback of the tautological bundle of $G(5, 10)$ and $\wh T_{ X }$ is the affine tangent bundle (see Equation \ref{eq:affine_tangent_bundle}). Then, as a main result, we have the following fact:
\begin{theorem}(Corollary \ref{cor:main})
    \label{thm:main}
    The exceptional collection of Equation \ref{eq:KP_collection_X} is full.
\end{theorem}
The proof is given by finding a sequence of mutations  (see Section \ref{subsec:mutations}) which relate such collection to the one found by Kuznetsov in \cite[Section 6.2]{KuznetsovHyperplane}. Since the latter is full, we conclude that also the KP collection is (Proposition \ref{prop:main}). 

\subsection*{Plan of the paper} In Section \ref{sec:vector_bundles} we give a brief introduction to the spinor tenfold and some vector bundles over such variety. Then, we establish resolutions of some homogeneous vector bundles (Lemma \ref{lem:affine_tangent_bundle_as_a_kernel} and Lemma \ref{lem:big_sequence}). In Section \ref{sec:Constructing_KP_collection} we perform the construction of the KP collection, and we prove its fullness in Section \ref{sec:fullness}. Finally, all the technical details of the cohomology computations based on the Borel--Bott--Weil theorem are carried out in Appendix \ref{sec:Borel--Bott--Weil}.

\subsection*{Acknowledgments} We would like to express our gratitude to Enrico Fatighenti, Micha\l\ Kapustka and Giovanni Mongardi for reading a first draft of this paper and providing valuable corrections and insights. Moreover, we thank Francesco Denisi, Jacopo Gandini and Luca Migliorini for helpful discussions and comments. The authors are members of GNSAGA of INdAM. MR is supported by PRIN2020KKWT53. 

\subsection{Notations and conventions}\label{subsec:notations_and_conventions}
    We shall work over the field of complex numbers.\\
    We will denote the cohomology of vector bundles as a direct sum of shifted cohomology spaces, and we will use the same notation for Ext spaces. For example, $\Ext_{\PP^n}^\bullet(\Oc, \Omega_{\PP^n}^1) \simeq \CC[-1]$.

\section{Vector bundles on the spinor tenfold}\label{sec:vector_bundles}


\subsection{The spinor tenfold and its properties}
Fix a vector space $V_n$ of dimension $n$. Let us call $G(k, V_n)$ the Grassmannian parametrizing $k$-dimensional subspaces of $V_n$. It is a smooth projective variety of dimension $k(n-k)$ projectively embedded in $\PP(\wedge^k V_n)$ by the Pl\"ucker embedding. 
By $OG(k, V_n)$ we denote the orthogonal Grassmannian of $k$-dimensional linear spaces, i.e. the smooth projective variety parametrizing affine $k$-spaces which are isotropic with respect to a fixed nondegenerate symmetric form on $V_n$. 
One can easily see that $OG(k, V_n)$ is cut by a general section of $\Sym^2\Uc^\vee$ in $G(k, V_n)$, where $\Uc$ is the tautological bundle. 
All orthogonal Grassmannians are connected except for $OG(n, V_{2n})$, which has two isomorphic connected components of dimension $n(n-1)/2$. 
These are the varieties parametrizing isotropic $n$-spaces which intersect a fixed $n$-space in even (respectively odd) dimension. As usual, we denote by \emph{spinor tenfold} each of the connected components of $OG(5, V_{10})$. 
In general, it is well-known that each connected component of $OG(n, V_{2n})$, also denoted by $OG(n, V_{2n})_\pm$, is isomorphic to $OG(n-1, V_{2n-1})$, where $V_{2n-1}$ is a $2n-1$-dimensional vector space. This fact can be proven with classical techniques, see \cite[Theorem 22.14]{HarrisAGFirst}. 

\subsubsection{A quick overview on homogeneous varieties}\label{subsubsec:homogeneous_varieties}
    The following is standard material, which we recall for the purpose of notation and self-containedness: for a more exhaustive treatment see, for instance, \cite{baston_eastwood_penrose_transform} and the sources therein. Let us consider a simple Lie group $G$, with Lie algebra $\mathfrak g$. Fix a Cartan subalgebra $\mathfrak h\subset\mathfrak g$ and call $\Delta$ the set of roots induced by its adjoint action, with $\Delta_+\subset \Delta$ the set of positive roots and $\Delta_- := \Delta\setminus\Delta_+$. One has a root space decomposition for $\mathfrak g$:
    \begin{equation*}
        \mathfrak g = \mathfrak h\oplus \bigoplus_{\alpha\in\Delta}\mathfrak g_\alpha
    \end{equation*}
    where the root eigenspaces $\mathfrak g_\alpha$ are given by
    \begin{equation*}
        \mathfrak g_\alpha = \{ g\in \mathfrak g : [h, g] = \alpha(h) g \,\,\,\text{for}\,\,\, h\in\mathfrak h \}
    \end{equation*} 
    Inside $\mathfrak g$, one distinguishes the \emph{standard Borel subalgebra} given by the expression
    \begin{equation*}
        \mathfrak b = \mathfrak h \oplus \mathfrak n
    \end{equation*}
    where $\mathfrak n = \bigoplus_{\alpha\in\Delta_+}\mathfrak g_\alpha$ is a nilpotent summand.\\
    \\
    A subalgebra $\mathfrak p\subset\mathfrak g$ is called parabolic if it contains $\mathfrak b$ as a subalgebra. Parabolic subalgebras can be characterized as follows. 
    Fix a set $\Sc=\{\alpha_1,\dots,\alpha_r\}$ of simple roots, define subsets $\Sc_{i_1\dots i_l}:= \Sc\setminus\{\alpha_{i_1},\dots,\alpha_{i_l}\}$ and $\Delta_{i_1\dots i_l} = \Span(\Sc_{i_1\dots i_l})\cap\Delta$. We consider the \emph{levi subalgebras}
    \begin{equation*}
        \mathfrak l_{i_1\dots i_l} = \mathfrak h\oplus\bigoplus_{\alpha\in\Delta_{i_1\dots i_l}}\mathfrak g_\alpha
    \end{equation*}
    and some nilpotent subalgebras of the form
    $$
    \mathfrak u_{{i_1\dots i_l}} = \bigoplus_{\alpha\in\Delta^+\setminus \Delta_{i_1\dots i_l}}\mathfrak g_\alpha
    $$
    Then the direct sum
    \begin{equation*}
        \mathfrak p_{{i_1\dots i_l}} = \mathfrak l_{{i_1\dots i_l}}\oplus\mathfrak u_{{i_1\dots i_l}}.
    \end{equation*}
    is a subalgebra and contains $\mathfrak b$, hence it is parabolic. The list of parabolic subalgebras $\{\mathfrak p_{i_1,\dots,i_l}\}$ is exhaustive up to conjugation.\\
    \\
    If $\mathfrak g$ is the Lie algebra of a simple Lie group $G$, it is well-known that a subgroup $P\subset G$ is parabolic if and only if its Lie algebra is a parabolic subalgebra. 
    Hence, in the following, we will associate to a rational homogeneous variety $G/P$ a marked Dynkin diagram.
    Such a diagram corresponds to the choice of $G$, while the markings correspond to the set $\Sc_{i_1\dots i_l}$ which fixes the parabolic subgroup (with respect to the choice of an ordering of simple roots). 
    We will also adopt the notation $G/P^{i_1, \dots, i_l}$ to describe the quotient of $G$ by the parabolic subgroup whose parabolic subalgebra is $\mathfrak p_{i_1,\dots, i_l}$.
    \subsubsection{Homogeneous vector bundles over \texorpdfstring{$G/P$}{something}}
    It is a general fact that one has an equivalence of categories:
    \begin{equation*}
        \coh^G(G/P)\simeq \operatorname{Rep}(P)
    \end{equation*}
    where on the left hand side we have the category of $G$-equivariant coherent sheaves on $G/P$ and on the right hand side there is the category of finite dimensional representations of $P$. 
    Note that we can identify irreducible representations of $P$ and $L$: in fact, any irreducible representation of $P$ is determined by its restriction to $L$, since the unipotent radical of $P$ acts trivially.
    Thus, a homogeneous, irreducible vector bundle $\Ec_\lambda$ on $G/P$ is uniquely determined by the highest weight $\lambda$ of a representation of $L$. More explicitly, we write:
    \begin{equation*}
        \Ec_\lambda := G\times^P \VV^P_\lambda
    \end{equation*}
    where $\VV^P_\lambda$ is the vector space on which the $P$-representation of highest weight $\omega$ acts, and with $G\times^P \VV^P_\lambda $ we denote the quotient of $G\times \VV^P_\lambda$ by the equivalence relation $(p.a, b) \sim (a, p.b)$ for every $p\in P$ and every $(a, b)\in G\times \VV^P_\lambda$. Let us call $W_L$ and $W_G$ the Weyl groups of respectively $L$ and $G$, and let us denote by $w_0^L$ and $w_0^G$ the respective longest elements. For later use, we recall that the Weyl reflections associated to simple roots act as follows:
    \begin{equation}
    \label{eq:weyl_reflection}
        S_{\alpha_i}: \alpha_j \longmapsto \alpha_j - A_{ij}\alpha_i
    \end{equation}
    where $A$ is the Cartan matrix. The following lemma is known to experts, see for instance \cite[Equation 8]{KP_collections_isotropic}, \cite[Lemma 2.1]{SmirnovTypeFArxiv}
    \begin{lemma}\label{lem:smirnov}
        One has:
        \begin{enumerate}
            \item $(\Ec_\lambda)^\vee \simeq \Ec_{-w_0^L\lambda}$
            \item if $\VV_\lambda^L\otimes \VV_\mu^L = \bigoplus \VV_\nu^L$, then $\Ec_\lambda\otimes \Ec_\mu = \bigoplus \Ec_\nu$.
        \end{enumerate}
    \end{lemma}
\subsubsection{Homogeneous description of \texorpdfstring{$OG(5, V_{10})_+$}{something}}\label{subsubsec:D5_description}
    Consider the simple group of type $D_5$ and its associated Dynkin diagram, with the ordering of the nodes chosen from left to right and from top to bottom for the last two (Bourbaki convention). We call $D_5/P^i$ the generalized Grassmannians of type $D_5$. In particular, the two connected components $OG(5, V_{10})_\pm$ of $OG(5, V_{10})$ can be identified, respectively, with $G/P^4$ and $G/P^5$. Hence, we can describe $OG(5, V_{10})_+$ with the following marked Dynkin diagram:
    $$
        \dynkin[root radius = 3pt, edge length = 25pt]{D}{ooo*o}
    $$
    The Levi subgroup is $SL(5)\times\CC^*$.
        Fixing an orthonormal basis $\{e_1, \dots, e_5\}$ of $\QQ^5$, one can write the simple roots and the fundamental weights as follows:
    \begin{equation*}
        \begin{split}
            \left.
            \begin{array}{cl}
                \alpha_1 & = e_1 - e_2 \\[2pt]
                \alpha_2 & = e_2 - e_3 \\[2pt]
                \alpha_3 & = e_3 - e_4 \\[2pt]
                \alpha_4 & = e_4 - e_5 \\[2pt]
                \alpha_5 & = e_4 + e_5. 
            \end{array}
            \right.
            & \hspace{60pt}
            \left.
            \begin{array}{cl}
                \omega_1 & = e_1 \\[2pt]
                \omega_2 & = e_1 + e_2 \\[2pt]
                \omega_3 & = e_1 + e_2 + e_3\\[2pt]
                \omega_4 & = \frac{1}{2}(e_1 + e_2 + e_3 + e_4 - e_5) \\[2pt]
                \omega_5 & = \frac{1}{2}(e_1 + e_2 + e_3 + e_4 + e_5). 
            \end{array}
            \right.
        \end{split}
    \end{equation*}
    Consider now a weight $\omega = \sum a_i\omega_i$. Then, by applying Equation \ref{eq:weyl_reflection} to the present setting, we find:
    \begin{equation*}
        \left.
        \begin{array}{cl}
                S_{\alpha_1}\omega &= -a_1\omega_1 + (a_2+a_1)\omega_2 + a_3\omega_3 + a_4\omega_4 + a_5\omega_5 \\[2pt]
                S_{\alpha_2}\omega &= (a_1+a_2)\omega_1 -a_2\omega_2 + a_3\omega_3 + a_4\omega_4 + a_5\omega_5 \\[2pt]
                S_{\alpha_3}\omega &= a_1\omega_1 + (a_2+a_3)\omega_2 -a_3\omega_3 + (a_4+a_3)\omega_4 + (a_5+a_3)\omega_5 \\[2pt]
                S_{\alpha_4}\omega &= a_1\omega_1 + a_2\omega_2 + (a_3+a_4)\omega_3 - a_4\omega_4 + a_5\omega_5 \\[2pt]
                S_{\alpha_5}\omega &= a_1\omega_1 + a_2\omega_2 + (a_3+a_5)\omega_3 + a_4\omega_4 - a_5\omega_5 \\
        \end{array}
        \right.
    \end{equation*}
    Call now $\Uc$ the pullback to $OG(5, V_{10})_+$ of the tautological bundle of $G(5, V_{10})$ under the natural embedding given by the fact that every element of $OG(5, V_{10})_+$ also belongs to $G(5, V_{10})$, and call $\Oc(1)$ the pullback of the hyperplane bundle of $\PP^{15}$ under the spinor embedding (see \cite[Section 6]{ranestad_schreyer_vsps} for an explicit description of such embedding in a local chart). Note that this is \emph{not} the pullback of $\Oc_{G(5, V_{10})}(1)$, which is $\Oc(2)$. One has, for every $k\in\ZZ$, for every positive $l\in\ZZ$ and for $m = 1,2,3$:
    \begin{equation}\label{eq:bundles_to_weights_D5}
            \Oc(k) = \Ec_{k\omega_4} \hspace{30pt} \Sym^l\Uc^\vee = \Ec_{l\omega_1} \hspace{30pt} \wedge^m\Uc^\vee = \Ec_{\omega_m} \hspace{30pt} \wedge^4\Uc^\vee = \Ec_{\omega_4+\omega_5}.
    \end{equation}
    Moreover, given a homogeneous irreducible vector bundle $\Ec_\lambda$, we have $\Ec_\lambda\otimes\Oc(k) = \Ec_{\lambda+k\omega_4}$. Observe that $\omega_1$ is dominant, and hence $\Uc^\vee$ is globally generated. Its sections are computed in Lemma \ref{lem:D5_dominant_weights} to be $H^\bullet(OG(5, V_{10})_+, \Uc^\vee) = \VV^{D_5}_{\omega_1}\simeq \CC^{10}$. On the other hand, $\Uc$ has no cohomology (Lemma \ref{lem:cohomology_U5}). We can identify the pullback of the tautological sequence of $G(5, V_{10})$ with the following:
    \begin{equation}\label{eq:tautological_sequence_rank_5}
        0\arw\Uc\arw \VV^{D_5}_{\omega_1}\otimes\Oc\arw \Uc^\vee\arw 0
    \end{equation}
    where we used the standard identification $\Uc^\vee\simeq (\VV^{D_5}_{\omega_1}\otimes\Oc)/\Uc$ and $\VV^{D_5}_{\omega_1}\simeq V_{10}$.
\subsubsection{Homogeneous description of $OG(4, V_{9})$}\label{subsubsec:B4_description}
    Let us carry out a similar analysis for $OG(4, V_{9})$. To distinguish from the previous case, let us call $Q^i$ the parabolic subgroups of the simple Lie group of type $B_4$, and $B_4/Q^i$ the associated generalized Grassmannians. If we consider the Dynkin diagram of type $B_4$, with the ordering of simple roots from left to right, we have $OG(4, V_{9}) = G/Q^4$, and the following marked Dynkin diagram:
    $$
        \dynkin[root radius = 3pt, edge length = 25pt]{B}{ooo*}
    $$
    In this case, the Levi subgroup is $SL(4)\times\CC^*$. With respect to an orthonormal basis $\{\epsilon_1, \dots, \epsilon_4\}$ of $\QQ^4$ we can choose the simple roots and the fundamental weights to be:
    \begin{equation*}
        \begin{split}
            \hspace{10pt}
            \left.
            \begin{array}{cl}
                \beta_1 & = \epsilon_1 - \epsilon_2 \\
                \beta_2 & = \epsilon_2 - \epsilon_3 \\
                \beta_3 & = \epsilon_3 - \epsilon_4 \\
                \beta_4 & = \epsilon_4. 
            \end{array}
            \right.
            & \hspace{40pt}
            \left.
            \begin{array}{cl}
                \omega_1 & = \epsilon_1 \\
                \omega_2 & = \epsilon_1 + \epsilon_2 \\
                \omega_3 & = \epsilon_1 + \epsilon_2 + \epsilon_3\\
                \omega_4 & = \epsilon_1 + \epsilon_2 + \epsilon_3 + 2\epsilon_4 
            \end{array}
            \right.
        \end{split}
    \end{equation*}
    By Equation \ref{eq:weyl_reflection}, for a weight $\nu = \sum b_i\nu_i$, the Weyl reflections associated to the simple roots $\{\beta_1,\dots, \beta_4\}$ are:
    \begin{equation*}
        \left.
        \begin{array}{cl}
                S_{\beta_1}\nu &= -b_1\nu_1 + (b_2+b_1)\nu_2 + b_3\nu_3 + b_4\nu_4\\[2pt]
                S_{\beta_2}\nu &= (b_1+b_2)\nu_1 -b_2\nu_2 + b_3\nu_3 + b_4\nu_4\\[2pt]
                S_{\beta_3}\nu &= b_1\nu_1 + (b_2+b_3)\nu_2 -b_3\nu_3 + (b_4+2b_3)\nu_4\\[2pt]
                S_{\beta_4}\nu &= b_1\nu_1 + b_2\nu_2 + (b_3+b_4)\nu_3 - b_4\nu_4
        \end{array}
        \right.
    \end{equation*}
    Denote by $\Rc$ the pullback to $OG(4, V_{9})$ of the tautological bundle of $G(4, V_9)$ and, as above, call $\Oc(1)$ the pullback of the hyperplane bundle of $\PP^{15}$. This again is not the pullback of $\Oc_{G(4, V_9)}(1)$, which is $\Oc(2)$. One has, for every $k\in\ZZ$, for every positive $l\in\ZZ$ and for $m = 1,2,3$:
    \begin{equation}\label{eq:bundles_to_weights_B4}
            \Oc(k) = \Ec_{k\nu_4} \hspace{30pt} \Sym^l\Rc^\vee = \Ec_{l\nu_1} \hspace{30pt} \wedge^m\Rc^\vee = \Ec_{\nu_m}.
    \end{equation}
    Pulling back the tautological sequence of $G(4, V_9)$ we find:
    \begin{equation}\label{eq:tautological_sequence_rank_4}
        0\arw\Rc\arw \VV^{B_4}_{\nu_1}\otimes\Oc\arw \Wc\arw 0
    \end{equation}
    where $\VV^{B_4}_{\nu_1}\simeq V_9$.
\subsection{Some exact sequences}
    The tautological bundles of the spinor tenfold with respect to its dual homogeneous description, and their Schur powers, interact in an interesting way giving rise to some exact sequences which will be useful in the next section. Let us describe how they can be constructed. Once for all, we set $X:=OG(5, V_{10})_+\simeq OG(4, V_{9})$.
    \begin{lemma}\label{lem:U4_and_U5}
        There is a short exact sequence on $X$:
        \begin{equation}
            \label{eq:sequence_of_tautologicals}
            0\arw\Rc\arw\Uc\arw\Oc\arw 0
        \end{equation}
        Moreover, one has an isomorphism $\Wc\simeq \Uc^\vee$.
    \end{lemma}
    \begin{proof}
        One has $\Ext^1(\Oc, \Rc) \simeq H^1( X , \Rc) \simeq \CC$ by Corollary \ref{cor:cohomology_U4}. Hence, up to isomorphisms, there exists a unique nontrivial extension of $\Rc$ by $\Oc$. Let us call $E$ such a nontrivial extension, and consider the short exact sequence:
        \begin{equation}\label{eq:extension_E_sequence}
            0\arw\Rc\arw E\arw \Oc\arw 0.
        \end{equation}
        By applying the functor $\Hom(-, \Uc)$ we have:
        \begin{equation*}
            0\arw\Hom(\Oc,\Uc)\arw\Hom(E,\Uc)\arw
            \Hom(\Rc, \Uc)\arw\Ext^1(\Oc, \Uc)\arw\cdots
        \end{equation*}
        and since by Lemma \ref{lem:cohomology_U5} the vector bundle $\Uc$ has no cohomology, the middle arrow is an isomorphism. Note that one can construct a nontrivial morphism $\Rc\arw\Uc$ by an embedding $V_9\subset V_{10}$ together with the fact that an isotropic $4$-space in $V_{10}$ can be completed to an isotropic $5$-space of $OG_+(5, 10)$ in exactly one way: hence there is at least a nonzero map $E \arw\Uc$. If both bundles are stable and they have the same slope, then the map must be an isomorphism: they have the same slope (both are extensions of $\Oc$ and $\Rc$), and $\Uc$ is stable because it is homogeneous and irreducible, let us prove that $E$ is stable as well. By Hoppe's criterion (see \cite{jardimmenetprataearp} and its application in \cite{manivel}), $E$ is stable if the following bundles have no global sections:
        $$
        E, \hspace{15pt} \wedge^2 E, \hspace{15pt} \wedge^3 E(1), \hspace{15pt} \wedge^4 E(1).
        $$
        Observe that, comparing the $r$-th and the $(r-1)$-rth exterior powers of the sequence \ref{eq:extension_E_sequence}, one finds a short exact sequence:
        \begin{equation*}
            0\arw \wedge^r \Rc\arw \wedge^r E\arw \wedge^{r-1}\Rc \arw 0.
        \end{equation*}
        By this and Lemma \ref{lem:vanishing_stability} we immediately see that $\wedge^2 E$, $\wedge^3 E(1)$ and $\wedge^4 E(1)$ have no sections (for the last we also use the fact that $\wedge^4 \Rc(1) \simeq \Oc(-1)$).\\
        \\
        Let us now prove that $E$ has no global sections. 
        Assume it had: then $H^0(X, E)\simeq H^1(X, E)\simeq H^0(X, \Oc)\simeq\CC$ because of the long exact sequence of cohomologies of the sequence \ref{eq:extension_E_sequence}. 
        Moreover, since by the sequence \ref{eq:extension_E_sequence} the top chern class of $E$ is zero, we have both an injective map $\Oc\arw E$ and a surjection $E\arw\Oc$ (the latter again from \ref{eq:extension_E_sequence}). The composition of these maps is an element of $\Hom(\Oc, \Oc)\simeq\CC$ and hence a multiple of the identity (nonvanishing by the exactness of \ref{eq:extension_E_sequence}): this implies that the sequence \ref{eq:extension_E_sequence} splits, providing a contradiction. 
        Thus, by Hoppe's criterion we conclude that $E$ is stable, and since it has the same slope of $\Uc$ and there exists a nonzero morphism $\Uc\arw E$, they are isomorphic, concluding the proof of the first claim.\\
        Let us now turn our attention to the second claim. One has a commutative diagram:
        \begin{equation*}
            \begin{tikzcd}
                0\ar{r} & \Rc\ar{r}\ar[hook]{d} & V_9\otimes\Oc\ar{r}\ar[hook]{d} & \Wc\ar{r}\ar{d}{\phi} & 0 \\
                0\ar{r} & \Uc\ar{r}\ar[maps to]{d} & V_{10}\otimes\Oc\ar{r}\ar[maps to]{d} & \Uc^\vee\ar{r} & 0 \\
                 & \Oc & \Oc & &
            \end{tikzcd}
        \end{equation*}
        and by the snake lemma, we conclude that $\phi$ is an isomorphism.
    \end{proof}
    The following lemma, and its proof, are already known to experts (see the MathOverflow thread \cite{mo_thread} for instance):
    \begin{lemma}\label{lem:tangent_bundle_as_an_extension}
        The tangent bundle $T_{ X }$ is the unique extension, up to isomorphism, fitting in the following exact sequence:
        \begin{equation}\label{eq:tangent_bundle_as_an_extension}
            0\arw\Rc^\vee\arw T_{ X }\arw \wedge^2 \Rc^\vee \arw 0.
        \end{equation}
        Moreover, one has $T_X\simeq\wedge^2\Uc^\vee$.
    \end{lemma}
    \begin{proof}
        Since $ X $ is the zero locus of a general section of $\Sym^2\Rc^\vee$ in $G(4, V_9)$, it comes with a normal bundle sequence:
        \begin{equation*}
            0\arw T_{ X }\arw\Rc^\vee\otimes\Wc\arw \Sym^2\Rc^\vee\arw 0.
        \end{equation*}
        Together with the sequence \ref{eq:sequence_of_tautologicals}, this gives rise to a diagram:
        \begin{equation*}
            \begin{tikzcd}
                 \coker\phi&0  & 0\\
                 T_{ X }\ar[hook]{r}\ar[two heads]{u} & \Rc^\vee\otimes\Wc \ar[two heads]{r}\ar{u} & \Sym^2\Rc^\vee\ar{u}\\
                 \Rc^\vee\ar[hook]{r}\ar{u}{\phi} & \Rc^\vee\otimes\Uc^\vee \ar[two heads]{r}\ar[equals]{u}& \Rc^\vee \otimes \Rc^\vee\ar[two heads]{u}\\
                 \ker\phi\ar[hook]{u} & 0\ar{u} & \wedge^2\Rc^\vee\ar[hook]{u}
            \end{tikzcd}
        \end{equation*}
        from which the existence of the sequence \ref{eq:tangent_bundle_as_an_extension} follows by the snake lemma. Uniqueness is a consequence of the following computation:
        \begin{equation*}
            \begin{split}
                \Ext^1(\wedge^2\Rc^\vee, \Rc^\vee) & \simeq H^1( X ,  \wedge^2\Rc\otimes\Rc^\vee) \\
                & \simeq H^1( X ,  \wedge^2\Rc^\vee\otimes\Rc^\vee(-2))\\
                & \simeq H^1( X ,  \wedge^3\Rc^\vee(-2))\oplus H^1( X ,  \Ec_{\nu_1+\nu_2-2\nu_4}).
            \end{split}
        \end{equation*}
        Here by Lemma \ref{lem:cohomology_extension_U4_O} the first summand is isomorphic to $\CC$ and the second is trivial, and this proves that there is only one isomorphism class of extensions of $\Rc^\vee$ and $\wedge^2\Rc^\vee$, concluding the proof of the first claim.\\
        To prove the second statement, let us call $\Uc_{G(5, V_{10})}$ and $\Qc_{G(5, V_{10})}$ the tautological and quotient bundles of $G(5, V_{10})$. Observe that $\Qc_{G(5, V_{10})}$ restricts to $\Uc^\vee$ on $X$, and therefore $T_{G(5, V_{10})}|_X\simeq \Uc^\vee\otimes\Uc^\vee$. Then, observe that there is a unique map $\Uc^\vee\otimes\Uc^\vee\arw \Sym^2\Uc^\vee$ up to scalar multiple, and such a map is surjective with kernel $\wedge^2\Uc^\vee$. Since $\Nc_{X|G(5, V_{10})} \simeq \Sym^2\Uc^\vee$, we conclude that $T_X\simeq\wedge^2\Uc^\vee$ by the normal bundle sequence of $X\subset G(5, V_{10})$.
    \end{proof}
    Recall the \emph{affine tangent bundle} $ \wh T_{ X }(1)$, which is an extension of the form
    \begin{equation}\label{eq:affine_tangent_bundle}
        0\arw\Oc(-1)\arw  \wh T_{ X } \arw T_{ X }(-1)\arw 0.
    \end{equation}
    Such extension is unique because $h^1( X , T_{ X }) = 1$ ($ X $ is rigid).
    \begin{lemma}
    \label{lem:affine_tangent_bundle_as_a_kernel}
        There is a short exact sequence:
        \begin{equation}
        \label{eq:affine_tangent_bundle_as_a_kernel}
            0\arw \wh T_{ X } \arw \VV^{D_5}_{\omega_4}\otimes\Oc\arw \Uc(1)\arw 0.
        \end{equation}
    \end{lemma}
    \begin{proof}
        By a simple calculation, it is easy to see that $\Uc(1)$ is globally generated. In fact, one has $\Uc(1)\simeq \wedge^4\Uc^\vee(1)$, and since $\wedge^4\Uc^\vee = \Ec_{\omega_4+\omega_5}$ and $\Oc(-1) = \Ec_{-\omega_5}$ we have $\Uc(1)\simeq\Ec_{\omega_4}$. Now, $\omega_4$ is a dominant weight and therefore $\Uc(1)$ is globally generated, with $H^0( X , \Uc(1)) = \VV^{D_5}_{\omega_4}\simeq \CC^{16}$. Let us define $\Fc_{11}$ as the kernel of the evaluation map on sections of $\Uc(1)$, i.e.
        \begin{equation}\label{eq:F11}
            \begin{tikzcd}
                0\ar{r} &    \Fc_{11}\ar{r} & H^0( X , \Uc(1))\otimes\Oc\ar{r}{\operatorname{ev}} & \Uc(1)\ar{r} & 0.
            \end{tikzcd}
        \end{equation}
        The rest of the proof consists in showing that $\wh T_{ X }\simeq \Fc_{11}(1)$.
        In the following diagram, the horizontal rows are respectively the sequence \ref{eq:F11} and the normal bundle sequence of $X\subset\PP^{15}$, while the central vertical row is the restriction to $X$ of the Euler sequence of the ambient $\PP^{15}$:
        \begin{equation*}
            \begin{tikzcd}
                 \ker\phi\ar[hook]{d}&\Oc\ar[hook]{d} & 0\ar{d} \\
                 \Fc_{11}(1)\ar[hook]{r}\ar{d}{\phi} & \VV^{D_5}_{\omega_4}\otimes\Oc(1) \ar[two heads]{r}\ar[two heads]{d} & \Uc(2)\ar[equals]{d}\\
                 T_{ X }\ar[hook]{r}\ar{d} & T_{\PP^{15}}|_{ X }\ar[two heads]{r}\ar{d} & \Uc(2)\ar{d}\\
                 \coker\phi & 0 & 0
            \end{tikzcd}
        \end{equation*}
        By the snake lemma we deduce that $\phi$ is surjective with kernel $\Oc$. Recall that extensions of the type described in the sequence \ref{eq:affine_tangent_bundle} are unique up to isomorphism: thus, $\wh T_{ X }\simeq \Fc_{11}(1)$.
    \end{proof}
    \begin{lemma}\label{lem:big_sequence}
        There is the following exact sequence:
        \begin{equation*}
            0\arw\Uc^\vee\arw \VV^{D_5}_{\omega_4}\otimes\Oc(1)\arw \VV^{D_5}_{\omega_1}\otimes\Uc(2)\arw \VV^{D_5}_{2\omega_1}\otimes\Oc(2)\arw \Sym^2\Uc^\vee(2)\arw 0
        \end{equation*}
    \end{lemma}
    \begin{proof}
        First, let us prove the existence of a sequence
        \begin{equation}\label{eq:partial_big_sequence}
            0\arw\wh T_X^\vee(-1)\arw \VV^{D_5}_{\omega_1}\otimes\Uc\arw \VV^{D_5}_{2\omega_1}\otimes\Oc\arw \Sym^2\Uc^\vee\arw 0.
        \end{equation}
        To do so, we start by applying the $\Sym^2$-Schur functor to the tautological exact sequence of $ X $ (as a $D_5$-Grassmannian), obtaining:
        \begin{equation*}
            0\arw\wedge^2\Uc\arw \VV^{D_5}_{\omega_1}\otimes\Uc\arw \Sym^2 \VV^{D_5}_{\omega_1}\otimes\Oc\arw \Sym^2\Uc^\vee\arw 0.
        \end{equation*}
        Note that $\Sym^2 \VV^{D_5}_{\omega_1} \simeq \VV^{D_5}_{2\omega_1}\oplus\CC$. Hence we write the following diagram, where rows and columns are exact, $\Kc$ and $\wt\Kc$ denote the kernels of the two rightmost horizontal maps and $\Cc$ is the cokernel of the central map of the second row:
        \begin{equation}\label{eq:big_big_diagram}
            \begin{tikzcd}
                & & \Kc\ar[hook]{rd} &\Oc\ar[hook]{d} & \\
                \wedge^2\Uc\ar[hook]{r}\ar[hook]{d} &  \VV^{D_5}_{\omega_1}\otimes\Uc\ar{rr}\ar[two heads]{ur}\ar[equals]{d} & & \VV^{D_5}_{2\omega_1}\otimes\Oc\oplus\Oc \ar[two heads]{r}\ar[two heads]{d} &  \Sym^2\Uc^\vee \\
                \wh T_X^\vee(-1) \ar[hook, dashed]{r}{\exists} \ar[two heads]{d} & \VV^{D_5}_{\omega_1}\otimes\Uc\ar{rr}\ar[two heads]{dr} & & \VV^{D_5}_{2\omega_1}\otimes\Oc \ar[two heads]{r} & \Cc \\
                \Oc & & \wt\Kc\ar[hook]{ru} & & \\
            \end{tikzcd}
        \end{equation}
        The dashed arrow, which makes the rightmost square commutative, exists because if we apply the functor $\Hom(-, \VV^{D_5}_{\omega_1}\otimes\Uc)$ to the leftmost column we obtain a long exact sequence of $\Ext$-groups, and since $\Ext^\bullet(\Oc, \Uc) = 0$ we find an isomorphism $\Hom(\wh T_X^\vee(-1), \VV^{D_5}_{\omega_1}\otimes\Uc) \simeq \Hom(\wedge^2\Uc, \VV^{D_5}_{\omega_1}\otimes\Uc)$. Then we can apply the snake lemma to the two leftmost vertical exact sequences, and this yields a short exact sequence
        \begin{equation*}
            0\arw \Oc\arw \Kc\arw \wt\Kc \arw 0.
        \end{equation*}
        Hence, we can write the following diagram:
        \begin{equation*}
            \begin{tikzcd}
                \Oc \ar[hook]{r}\ar[equals]{d}& \Kc\ar[two heads]{r} \ar[hook]{d} & \wt\Kc \ar[hook]{d} \\
                \Oc \ar[hook]{r}& \VV^{D_5}_{2\omega_1}\otimes\Oc\oplus\Oc \ar[two heads]{r} \ar[two heads]{d} & \VV^{D_5}_{2\omega_1}\otimes\Oc \ar[two heads]{d} \\
                 & \Sym^2\Uc^\vee & \Cc
            \end{tikzcd}
        \end{equation*}
        By the snake lemma we conclude that $\Cc\simeq \Sym^2\Uc^\vee$, which gives us the sequence \ref{eq:partial_big_sequence}.\\
        By Lemma \ref{lem:ugly_ext} one has $\Ext(\Sym^2\Uc^\vee,    \wh T_{ X }^\vee(-1)) = \CC[-2]$: hence, there exists a unique two-terms extension of the form
        \begin{equation*}
            0\arw    \wh T_{ X }^\vee(1)\arw B\arw C\arw \Sym^2\Uc^\vee(2) \arw 0.
        \end{equation*}
        On the other hand, we see that $\Ext(\wh T_X^\vee(1),\Uc^\vee) = \CC[-1]$: this follows by the fact that by twisting the sequence \ref{eq:affine_tangent_bundle_as_a_kernel} we obtain:
        \begin{equation*}
            0\arw \Uc^\vee\otimes \wh T_X(-1)\arw \VV^{D_5}_{\omega_1}\otimes\Uc^\vee(-1)\arw\Uc^\vee\otimes\Uc\arw 0,
        \end{equation*}
        where the middle term has no cohomology by Lemma \ref{lem:coh_general_D5} and the last one contributes with $\CC[0]$ because $\Uc$ is exceptional. Thus, there is a unique $A$, up to isomorphism, fitting in the following exact sequence:
        \begin{equation*}
            0\arw \Uc^\vee\arw A \arw \wh T_X^\vee(1) \arw 0.
        \end{equation*}
        By composing the last two equations, respectively, with appropriate twists of the duals of \ref{eq:F11} and \ref{eq:affine_tangent_bundle}, we find two long exact sequences:
        \begin{equation}\label{eq:three_terms_extension_1}
            0\arw \Uc^\vee\arw \VV^{D_5}_{\omega_4}\otimes\Oc(1)\arw B\arw C\arw \Sym^2\Uc^\vee(2) \arw 0.
        \end{equation}
        \begin{equation}\label{eq:three_terms_extension_2}
            0\arw\Uc^\vee\arw A \arw \VV^{D_5}_{\omega_1}\otimes\Uc(2)\arw \VV^{D_5}_{2\omega_1}\otimes\Oc(2)\arw \Sym^2\Uc^\vee(2)\arw 0
        \end{equation}
        Now, by Lemma \ref{lem:ext_c3} there is a unique three-terms extension between $\Uc^\vee$ and $\Sym^2\Uc^\vee(2)$ which is not a direct sum of shorter extensions, and therefore the sequences \ref{eq:three_terms_extension_1} and \ref{eq:three_terms_extension_2} must coincide up to isomorphisms, and this concludes the proof.
    \end{proof}
    
    \section{The KP collection on the spinor tenfold}\label{sec:Constructing_KP_collection}
    Let $G$ be a simple Lie group of type $B_4$. By $\dbcoh_G(X)$ we denote the equivariant category of coherent sheaves on the spinor tenfold $X$. We will use the notation\footnote{Note that this choice of notation differs from the one of Kuznetsov and Polishchuk.} $\Ec_\lambda$ for both the homogeneous vector bundle $G\times^P \VV^P_\lambda$ in the category $\dbcoh_G(X)$ of equivariant sheaves, and its image through the forgetful functor $\forgetful : \dbcoh_G(X)\arw \dbcoh(X)$. 
    We also fix the shorthand notation:
    \begin{equation*}
        \begin{split}
            \Ext_G^\bullet(-,-)& :=\Ext_{\dbcoh_G(X)}^\bullet(-,-)\\
            \Ext^\bullet(-,-) & := \Ext^\bullet_{\dbcoh(X)}(-,-).
        \end{split}
    \end{equation*}
    As usual, given a vector space $W$ with a $G$-action, we call $W^G\subset W$ the space of $G$-invariants.\\
    By the property $\Ext_G^\bullet (\Ec_\lambda, \Ec_\mu) = \Ext^\bullet(\Ec_\lambda, \Ec_\mu)^G$ \cite[Proposition 2.17]{KP_collections_isotropic}, one has the remarkable property that every homogeneous, irreducible vector bundle is exceptional in $\dbcoh_G(X)$. For the standard definitions of semiorthogonal decomposition, exceptional object and full exceptional collection, we refer to \cite[Section 2]{kuznetsov_derived_category_view}.
    \subsection{Mutations of exceptional objects}\label{subsec:mutations}
    Let $\Tc$ be a triangulated category, and $\Ac_1, \dots, \Ac_n$ admissible subcategories such that $\Tc = \langle\Ac_1, \dots, \Ac_n\rangle$ is a semiorthogonal decomposition. Mutations are a well-established technique for producing other semiorthogonal decompositions from a given one (see \cite{bondal_associative_algebras, bondalorlov}). The setting of the problem we address in this paper is much more specific: let $Y$ be a smooth projective variety and $\dbcoh(Y) = \langle E_1, \dots, E_n\rangle$ a full exceptional collection of vector bundles. Then, all the techniques we need to recall about mutations can be summarized in the following two lemmas, which are well-known to experts.
    \begin{lemma}\label{lem:mutations_are_cones}
        Let $Y$ be a smooth projective variety and $\dbcoh(Y) = \langle E_1, \dots, E_n\rangle$ a full exceptional collection. Then, one has:
        \begin{equation*}
            \begin{split}
                \dbcoh(Y) & = \langle E_1, \dots, E_{i-1}, E_{i+1}, \RR_{E_{i+1}}E_i, E_{i+2}, \dots, E_n\rangle \\
                & = \langle E_1, \dots, E_{i-1}, \LL_{E_i}E_{i+1}, E_i, E_{i+2}, \dots, E_n\rangle
            \end{split}
        \end{equation*}
        where:
        \begin{equation*}
            \begin{split}
                \RR_{E_{i+1}}E_i & = \cone(E_i\arw E_{i+1}\otimes\Ext_\Tc ^\bullet(E_i, E_{i+1})^\vee)[-1]\\
                \LL_{E_i}E_{i+1} & = \cone(E_i\otimes\Ext_\Tc ^\bullet(E_i, E_{i+1})\arw E_{i+1})
            \end{split}
        \end{equation*}
        with the morphisms being respectively the evaluation and coevaluation maps.
    \end{lemma}
    \begin{lemma}\label{lem:mutations_recipe}
        Consider a full exceptional collection $\dbcoh(Y) = \langle E_1, \dots, E_n\rangle$ for $Y$ smooth projective. Then, up to a shift, the following holds:
        \begin{enumerate}
            \item If there is a short exact sequence $0\arw F\arw V\otimes E_1 \arw E_2\arw 0$, and $\Ext_{\dbcoh(Y)}^\bullet(E_1, E_2) = V[0]$, then $\LL_{E_1}E_2 \simeq F$.
            \item If there is a short exact sequence $0\arw E_1 \arw V\otimes E_2\arw F \arw 0$, and $\Ext_{\dbcoh(Y)}^\bullet(E_1, E_2) = V[0]$, then $\RR_{E_2}E_1 \simeq F$.
            \item If there is a short exact sequence $0\arw E_2 \arw F \arw E_1\arw 0$, and $\Ext_{\dbcoh(Y)}^\bullet(E_1, E_2) = \CC[-1]$, then $\RR_{E_2}E_1 \simeq \LL_{E_1}E_2 \simeq F$.
        \end{enumerate}
    \end{lemma}
    We also recall, in our setting, the definition of a right dual exceptional collection (as given in \cite{KP_collections_isotropic}).
    \begin{definition}\label{def:right_dual}
        Given $\dbcoh(Y) = \langle E_1, \dots, E_n\rangle$ as above, we define the associated $\emph{right dual}$ collection as follows:
        \begin{equation*}
            \langle E_n, \RR_{E_n}E_{n-1}, \dots, \RR_{\
            \langle E_2,\dots, E_n\rangle}E_1 \rangle.
        \end{equation*}
    \end{definition}
    \subsection{The forgetful functor and the exceptional blocks}
    Let us apply the prescription of \cite[Theorem 9.3]{KP_collections_isotropic} to $G/P =  X $ to produce a list of subcategories $\Ac_0, \dots, \Ac_7$ of $\dbcoh_G(X)$, each generated by exceptional objects. We find:
    \begin{equation*}
        \begin{split}
            \Ac_0 &= \langle\Oc\rangle_G \\
            \Ac_1 &= \langle\Oc(1)\rangle_G \\
            \Ac_2 &= \langle\Sym^2\Rc^\vee(2), \Rc^\vee(2), \Oc(2)\rangle_G \\
            \Ac_3 &= \langle\Sym^2\Rc^\vee(3), \Rc^\vee(3), \Oc(3)\rangle_G \\
            \Ac_4 &= \langle\wedge^2\Rc^\vee(4), \Rc^\vee(4), \Oc(4)\rangle_G \\
            \Ac_5 &= \langle\wedge^2\Rc^\vee(5), \Rc^\vee(5), \Oc(5)\rangle_G \\
            \Ac_6 &= \langle\Oc(6)\rangle_G \\
            \Ac_7 &= \langle\Oc(7)\rangle_G \\
        \end{split}
    \end{equation*}
    where $\langle \quad\rangle_G$ represents an exceptional collection in $\dbcoh_G( X )$. Note that a na\"ive application of the forgetful functor on each generator of the $\Ac_i$'s does \emph{not} produce an exceptional collection in $\dbcoh( X )$. 
    According to \cite[Section 3.3]{KP_collections_isotropic}, a possible way to produce an exceptional collection for $X$ is the following: first we take the right dual $\Bc_i$ of the block $\Ac_i$, and we apply the forgetful functor to each generator of each block $\Bc_i$. Finally, we consider the collection $\langle \forgetful\Bc_0, \dots, \forgetful\Bc_7\rangle$.
    This task is carried out in the following of this section.
    \begin{lemma}
        There is the following equivalence of categories:
        \begin{equation*}
            \forgetful \langle\wedge^2\Rc^\vee, \Rc^\vee, \Oc\rangle_G \simeq \langle \Oc, \Uc^\vee,    \wh T_{ X }(1) \rangle.
        \end{equation*}
    \end{lemma}
    \begin{proof}
        Following \cite[Section 3.3]{KP_collections_isotropic} we can produce a full exceptional collection for $\forgetful \langle\wedge^2\Rc^\vee, \Rc^\vee, \Oc\rangle_G$ by constructing the right dual collection to $\langle\wedge^2\Rc^\vee, \Rc^\vee, \Oc\rangle_G$, \emph{in the equivariant category}, and then applying the forgetful functor on each generator. In other words, we only need to show that the right dual of $\langle\wedge^2\Rc^\vee, \Rc^\vee, \Oc\rangle_G$ is $\langle \Oc, \Uc^\vee, \wh T_{ X }(1) \rangle_G$. Observe that, by Lemma \ref{lem:tangent_bundle_as_an_extension} and \ref{lem:equivariant_extension_1} we can apply Lemma \ref{lem:mutations_recipe} to obtain $\RR_{\Rc^\vee}\wedge^2\Rc^\vee =  T_X$.
        Now, let us mutate the result one step further to the right. The relevant Ext is $\Ext_G^\bullet( T_X, \Oc) = \CC[-1]$. Then, in light of the sequence \ref{eq:affine_tangent_bundle}, we conclude by Lemma \ref{lem:mutations_recipe} that $\RR_{\langle\Rc^\vee, \Oc\rangle}\wedge^2\Rc^\vee \simeq    \wh T_{ X }(1)$. Similarly, by Lemma \ref{lem:U4_and_U5} (and in particular by the computations in its proof) we have $\RR_\Oc\Rc^\vee \simeq\Uc^\vee$. This concludes the proof.
        \end{proof}
        \begin{lemma}
        There is the following equivalence of categories:
        \begin{equation*}
            \forgetful \langle\Sym^2\Rc^\vee, \Rc^\vee, \Oc\rangle_G \simeq \langle \Oc, \Uc^\vee, \Sym^2\Uc^\vee\rangle.
        \end{equation*}
    \end{lemma}
    \begin{proof}
        As above, we only need to prove that the right dual of $\langle\Sym^2\Rc^\vee, \Rc^\vee, \Oc\rangle_G$ is $\langle \Oc, \Uc^\vee, \Sym^2\Uc^\vee\rangle_G$. Recall that mutations do not depend on the choice of the exceptional collection of the category we are mutating through. Hence, by the proof of the previous lemma, we have $\RR_{\langle \Rc^\vee, \Oc\rangle}\Sym^2\Rc^\vee = \RR_{\langle \Oc, \Uc^\vee\rangle}\Sym^2\Rc^\vee$. 
        The right hand side of this last equation can be computed as follows: we first note that $\Ext_G(\Sym^2\Rc^\vee, \Oc) = 0$. In fact, one has $\Ext_G(\Sym^2\Rc^\vee, \Oc) = H^\bullet(X, \Sym^2\Rc)^G$, and the latter vanishes since $H^\bullet(X, \Sym^2\Rc)$ by Lemma \ref{lem:vanishing_sym2U4}. Finally, we are left to compute $\RR_{\Uc^\vee}\Sym^2\Rc^\vee$.\\
        By taking the second exterior power of the dual of the sequence \ref{eq:sequence_of_tautologicals} we get:
        \begin{equation*}
            0\arw \Uc^\vee\arw \Sym^2\Uc^\vee\arw \Sym^2\Rc^\vee\arw 0,
        \end{equation*}
        and since one has $\Ext_G^\bullet(\Sym^2\Rc^\vee, \Uc^\vee) = H^\bullet(X, \Sym^2\Rc\otimes\Uc^\vee)^G = \CC[-1]$ by Lemma \ref{lem:vanishing_mixed}, we are able to conclude that $\RR_{\Uc^\vee}\Sym^2\Rc^\vee \simeq \Sym^2\Uc^\vee$.
    \end{proof}
    Summing all up, we find:
    \begin{proposition}\label{prop:KP_collection_X}
        The following collection is exceptional in $\dbcoh( X )$:
        \begin{equation*}
            \begin{split}
                \Ac = \langle &\Oc, \Oc(1), \Oc(2), \Uc^\vee(2), \Sym^2\Uc^\vee(2), \Oc(3), \Uc^\vee(3), \Sym^2\Uc^\vee(3), \\
                & \hspace{40pt}\Oc(4), \Uc^\vee(4),    \wh T_{ X }(5), \Oc(5), \Uc^\vee(5),    \wh T_{ X }(6), \Oc(6), \Oc(7) \rangle
            \end{split}
        \end{equation*}
    \end{proposition}
    \begin{remark}
        The presence of the affine tangent bundle in exceptional collections of homogeneous varieties has been investigated in \cite[Observation 1.4]{SmirnovTypeFArxiv}, where the author observes the relation with the fact that the tangent bundle is a homogeneous, irreducible vector bundle, and therefore it is exceptional in the equivariant category. This fact is true only for cominuscule varieties (like the spinor tenfold).
    \end{remark}
    
    \subsection{Fullness}\label{sec:fullness}
    Hereafter we describe the main result of this paper (Theorem \ref{prop:main}), which is the fullness of the KP collection for $X$ we constructed in Proposition \ref{prop:KP_collection_X}. In particular, we describe a sequence of mutations which, applied to the KP collection, yield the full exceptional collection constructed by Kuznetsov in \cite{KuznetsovHyperplane}, thus proving that the KP collection is full as well.
    \begin{proposition}
        \label{prop:main}
        There is an equivalence of categories:
        \begin{equation*}
            \begin{split}
                \langle &\Oc, \Oc(1), \Oc(2), \Uc^\vee(2), \Sym^2\Uc^\vee(2), \Oc(3), \Uc^\vee(3), \Sym^2\Uc^\vee(3), \\
                & \hspace{40pt}\Oc(4), \Uc^\vee(4),    \wh T_{ X }(5), \Oc(5), \Uc^\vee(5),    \wh T_{ X }(6), \Oc(6), \Oc(7) \rangle \\
                \simeq \langle &\Oc, \Uc^\vee, \Oc(1), \Uc^\vee(1), \Oc(2), \Uc^\vee(2), \Oc(3), \Uc^\vee(3), \\
                & \hspace{40pt}\Oc(4), \Uc^\vee(4), \Oc(5), \Uc^\vee(5), \Oc(6), \Uc^\vee(6), \Oc(7), \Uc^\vee(7)\rangle \\
            \end{split}
        \end{equation*}
    \end{proposition}
    \begin{proof}
    Let us first describe how we can mutate $   \wh T_{ X }(6)$ one step to the right.
    In light of Lemmas \ref{lem:mutations_recipe} and \ref{lem:affine_tangent_bundle_as_a_kernel}, we have $\RR_{\Oc(6)}   \wh T_{ X }(6)\simeq \Uc^\vee(7)$ if we can prove that $\Ext^\bullet(   \wh T_{ X }(6), \Oc(6))\simeq \VV^{D_5}_{\omega_4}[0]$, and the latter follows by Corollary \ref{cor:sections_dual_affine_tangent_bundle} once we note that $\Ext^\bullet(   \wh T_{ X }(6), \Oc(6))\simeq H^\bullet(X, \wh T_X^\vee)$. We can apply the same exact strategy to mutate $   \wh T_{ X }(5)$ one step to the right, obtaining:
    \begin{equation*}
            \begin{split}
                \langle &\Oc, \Oc(1), \Oc(2), \Uc^\vee(2), \Sym^2\Uc^\vee(2), \Oc(3), \Uc^\vee(3), \Sym^2\Uc^\vee(3), \\
                & \hspace{40pt}\Oc(4), \Uc^\vee(4), \Oc(5), \Uc(6), \Uc^\vee(5), \Oc(6), \Uc(7), \Oc(7) \rangle.
            \end{split}
        \end{equation*}
    Now, by Corollary \ref{cor:sections_symmetric_powers}, the sequence \ref{eq:tautological_sequence_rank_5} and Lemma \ref{lem:mutations_recipe} we easily see that $\LL_{\Oc(k)}\Uc^\vee(k) \simeq \Uc(k)$ for any $k$. Let us apply it to $k=2$ in the collection above. Then, note that by Corollary \ref{cor:sections_symmetric_powers} one has:
    \begin{equation*}
        \Ext^\bullet(\Oc(2), \Sym^2\Uc^\vee(2)) = \VV^{D_5}_{2\omega_1}[0]
    \end{equation*}
    and therefore, by Lemma \ref{lem:mutations_recipe} one has $\LL_{\Oc(2)} \Sym^2\Uc^\vee(2) = \wt\Kc(2) = \ker(\VV^{D_5}_{2\omega_1}\otimes\Oc(2)\arw \Sym^2\Uc^\vee(2))$ (see the proof of Lemma \ref{lem:big_sequence} for more on this bundle). Now, we claim that $\LL_{\Uc(2)}\wt\Kc(2) = \wh T_X^\vee(1)$. In fact, $\Ext^\bullet(\Uc(2), \wt\Kc(2))$ can be computed by applying $\Ext(\Uc(2), -)$ to the sequence describing $\wt\Kc(2)$ as a kernel (see Diagram \ref{eq:big_big_diagram}): in particular by Corollary \ref{cor:sections_symmetric_powers} one has $\Ext^\bullet(\Uc(2), \Oc(2)) = \VV^{D_5}_{\omega_1}[0]$, and, together with Corollary \ref{cor:some_bott_computation}, this allows to conclude that $
    \Ext^\bullet(\Uc(2), \wt\Kc(2)) \simeq \VV^{D_5}_{\omega_1}[0]$. From this and Lemma \ref{lem:mutations_recipe} the claim follows. Let us now mutate the resulting $\wh T_X^\vee(1)$ a further step to the left. We have $\Ext^\bullet(\Oc(1), \wh T_X^\vee(1)) \simeq V_{\omega_4}[0]$, and therefore, by the appropriate twist of the dual of Equation \ref{eq:affine_tangent_bundle_as_a_kernel}, we find $\LL_{\Oc(1)}\wh T_X^\vee(1) \simeq \Uc^\vee$.
    If we repeat the same exact procedure on $\Sym^2\Uc^\vee(3)$ we get the collection:
    \begin{equation*}
            \begin{split}
                \langle &\Oc, \Uc^\vee, \Oc(1), \Uc^\vee(1), \Oc(2), \Uc^\vee(2), \Oc(3), \Uc^\vee(3), \\
                & \hspace{40pt}\Oc(4), \Uc^\vee(4), \Oc(5), \Uc(6), \Uc^\vee(5), \Oc(6), \Uc(7), \Oc(7) \rangle.
            \end{split}
    \end{equation*}
    Since by Lemma \ref{lem:orthogonals_in_the_collection} we have $\Ext(\Uc(6), \Uc^\vee(5)) = 0$, we can freely exchange such objects. The last step is to mutate each bundle of the form $\Uc(k)$ one step to the right. One has $\Ext^\bullet(\Uc(k), \Oc(k)) = \VV^{D_5}_{\omega_1}[0]$, and by the sequence \ref{eq:tautological_sequence_rank_5} and Lemma \ref{lem:mutations_recipe} we obtain $\RR_{\Oc(k)}\Uc(k) \simeq\Uc^\vee(k)$. This leads to Kuznetsov's collection, hence concluding our proof.
    \end{proof}
    The immediate consequence of Proposition \ref{prop:main} is the following:
    \begin{corollary}(Theorem \ref{thm:main})
    \label{cor:main}
        The exceptional collection of Equation \ref{eq:KP_collection_X} is full.
    \end{corollary}
    \appendix
    
    \section{Borel--Bott--Weil computations}\label{sec:Borel--Bott--Weil}
    
    In this appendix, we will gather all cohomological computations done by  means of the Borel--Bott--Weil theorem \cite{bott}. This  well-known theorem provides a simple, algorithmic recipe to compute the cohomology of irreducible vector bundles.
    \begin{theorem}[Borel--Bott--Weil]\label{thm:borel_weil_bott}
        Let $\Ec_\omega$ be a homogeneous, irreducible vector bundle on a rational homogeneous variety $G/P$, let us call $\rho$ the sum of all fundamental weights. Then one and only one of the following statements is true:
        \begin{enumerate}
            \item there exists a sequence $s_p$ of simple Weyl reflections of length $p$ such that $s_p(\omega+\rho) - \rho$ is dominant (i.e. all coefficients of its expansion in the fundamental weights are non-negative). Then $H^p(G/P, \Ec_\omega)\simeq V^G_{s_p(\omega+\rho) - \rho}$ and all the other cohomology is trivial.
            \item  there is no such $s_p$ as in point (1). Then $\Ec_\omega$ has no cohomology.
        \end{enumerate}
    \end{theorem}
    While we chose to illustrate all computations in detail, the algorithm can be easily automatized. See, for example, the script \cite{python_script}. 
    \subsection{Cohomology computations on \texorpdfstring{$D_5/P^4$}{something}}\label{subsec:bwb_D5}
    Here we follow the notation we fixed in Section \ref{subsubsec:D5_description}.
    \begin{lemma}\label{lem:D5_dominant_weights}
        For every nonnegatgive $a,b,c,d\in\ZZ$, one has $H^\bullet(X, \Ec_{a\omega_1+b\omega_2+c\omega_3+d\omega_5}) = \VV^{D_5}_{a\omega_1+b\omega_2+c\omega_3+d\omega_5}[0]$.
    \end{lemma}
    \begin{proof}
        The weight $a\omega_1+b\omega_2+c\omega_3+d\omega_5$ is already dominant: by Theorem \ref{thm:borel_weil_bott} we immediately conclude.
    \end{proof}
    \begin{corollary}\label{cor:sections_symmetric_powers}
        For every nonnegative integer $r$, one has $H^\bullet(X, \Sym^r\Uc^\vee) = \VV^{D_5}_{r\omega_1}[0]$.
    \end{corollary}
    \begin{proof}
        The weight associated to $\Sym^r\Uc^\vee$ is $r\omega_1$ (see Equation \ref{eq:bundles_to_weights_D5}), hence we conclude by Lemma \ref{lem:D5_dominant_weights}.
    \end{proof}
    \begin{corollary}\label{cor:some_bott_computation}
        There are the following isomorphisms:
        \begin{equation*}
            \Ext^\bullet(\Uc(2), \Sym^2\Uc^\vee(2)) \simeq H^\bullet(X, \Sym^3\Uc^\vee\oplus\Ec_{\omega_1+\omega_2})\simeq \VV^{D_5}_{3\omega_1}[0]\oplus\VV^{D_5}_{\omega_1+\omega_2}[0].
        \end{equation*}
    \end{corollary}
    \begin{proof}
        The first step is simply Lemma \ref{lem:smirnov}, then we conclude by Corollary \ref{cor:sections_symmetric_powers}.
    \end{proof}
    \begin{lemma}\label{lem:coh_general_D5}
        For all integers $a\geq 0$, $b\geq 0$, $c>0$ and for $\epsilon\in\{1,2\}$ one has:
        \begin{equation*}
            \begin{split}
                H^\bullet(X, \Ec_{a\omega_1+b\omega_2-\epsilon\omega_4}) & = 0 \\
                H^\bullet(X, \Ec_{a\omega_1+b\omega_2+ c\omega_3-2\omega_4}) & = \VV^{D_5}_{a\omega_1+b\omega_2+(c-1)\omega_3}[-1].
            \end{split}
        \end{equation*}
    \end{lemma}
    \begin{proof}
        An intuitive way to apply the Borel--Bott--Weil algorithm is to write the coefficients of the weight of the bundle as labels for the nodes of the Dynkin diagram, then add the sum $\rho$ of the fundamental weights, and then to apply the Weyl reflections associated to simple roots. By adding $\rho$ to the weight $a\omega_1+b\omega_2-\epsilon\omega_4$ we find:
        \begin{equation*}
            \dynkin[root radius = 3pt, edge length = 25pt,labels*={a, b, 0, -\epsilon, 0}]{D}{ooo*o} \xrightarrow{\hspace{10pt} +\rho \hspace{10pt}}\dynkin[root radius = 3pt, edge length = 25pt,labels*={a+1, b+1, 1, 1-\epsilon, 1}]{D}{ooo*o}
        \end{equation*}
        Observe that for $\epsilon = 1$ we have a zero coefficient: hence, the weight is fixed by the action of $S_{\alpha_4}$ and there is no finite number of simple Weyl reflections which transforms it to a weight of the form $\rho + \lambda$ with $\lambda$ dominant. Thus, the associated bundle has no cohomology by Theorem \ref{thm:borel_weil_bott}.
        On the other hand, if $\epsilon = 2$, one has
        \begin{equation*}
            \dynkin[root radius = 3pt, edge length = 25pt,labels*={a+1, b+1, 1, -1, 1}]{D}{ooo*o} \xrightarrow{\hspace{10pt} S_{\alpha_4} \hspace{10pt}}\dynkin[root radius = 3pt, edge length = 25pt,labels*={a+1, b+1, 0, 1, 1}]{D}{ooo*o}
        \end{equation*}
        and the cohomology vanishes for the same reason as above.
        Let us now consider the second claim. By adding $\rho$ to $a\omega_1+b\omega_2+ c\omega_3-2\omega_4$ we obtain:
        \begin{equation*}
            \dynkin[root radius = 3pt, edge length = 25pt,labels*={a, b, c, -2, 0}]{D}{ooo*o} \xrightarrow{\hspace{10pt} +\rho \hspace{10pt}}\dynkin[root radius = 3pt, edge length = 25pt,labels*={a+1, , c+1, -1, 1}, labels={,b+1,,,}]{D}{ooo*o}
        \end{equation*}
        Then, applying $S_{\omega_4}$ yields:
        \begin{equation*}
            \dynkin[root radius = 3pt, edge length = 25pt,labels*={a+1, , c+1, -1, 1}, labels={,b+1,,,}]{D}{ooo*o} \xrightarrow{\hspace{10pt} S_{\alpha_4} \hspace{10pt}}\dynkin[root radius = 3pt, edge length = 25pt,labels*={a+1, b+1, c, 1, 1}]{D}{ooo*o}
        \end{equation*}
        Note that if we subtract $\rho$ to the resulting weight we obtain the dominant weight $a\omega_1+b\omega_2+(c-1)\omega_3$, and the result follows again by Theorem \ref{thm:borel_weil_bott}.
    \end{proof}
    \begin{corollary} \label{cor:sections_dual_affine_tangent_bundle}
        One has $H^\bullet(X, \wh T_X^\vee) \simeq \VV^{D_5}_{\omega_4}[0]$.
    \end{corollary}
    \begin{proof}
        Consider the dual of the sequence \ref{eq:affine_tangent_bundle_as_a_kernel}. The claim follows once we observe that by Lemma \ref{lem:coh_general_D5} the bundle $\Uc^\vee(-1) = \Ec_{\omega_1-\omega_4}$ has no cohomology.
    \end{proof}
    \begin{lemma}\label{lem:cohomology_U5}
       The vector bundle $\Uc$ has no cohomology.
    \end{lemma}
    \begin{proof}
        Since $\Uc\simeq\wedge^4\Uc(-2) \simeq \Ec_{-\omega_4+\omega_5}$, the proof is immediate: one has
        \begin{equation*}
            \dynkin[root radius = 3pt, edge length = 25pt,labels*={0, 0, 0, -1, 1}]{D}{ooo*o} \xrightarrow{\hspace{10pt} +\rho \hspace{10pt}}\dynkin[root radius = 3pt, edge length = 25pt,labels*={1, 1, 1, 0, 2}]{D}{ooo*o}
        \end{equation*}
        and this allows to conclude as we did in the proof of Lemma \ref{lem:coh_general_D5}.
    \end{proof}
    \begin{lemma}
    \label{lem:orthogonals_in_the_collection}
        We have $\Ext^\bullet(\Uc(1), \Uc^\vee) = 0$. 
    \end{lemma}
    \begin{proof}
        We begin by decomposing the relevant bundle in a sum of irreducible, by Lemma \ref{lem:smirnov}:
        \begin{equation*}
            \begin{split}
                \Ext^\bullet(\Uc(1), \Uc^\vee) &= H^\bullet(X, \Uc^\vee\otimes\Uc^\vee(-1)) \\
                &= H^\bullet(X, \Sym^2\Uc^\vee(-1)\oplus\wedge^2\Uc^\vee(-1))\\
                &=H^\bullet(X, \Ec_{2\omega_1-\omega_4}\oplus\Ec_{\omega_2-\omega_4}).
            \end{split}
        \end{equation*}
        Then, the proof follows by applying Lemma \ref{lem:coh_general_D5}.
    \end{proof}
    \begin{lemma}\label{lem:ext_c3}
        One has $\Ext(\Sym^2\Uc^\vee(2), \Uc^\vee) = \CC[-3]$.
    \end{lemma}
    \begin{proof}
        One possibility is to expand the tensor product $\Sym^2\Uc\otimes\Uc^\vee(-2)$ in irreducibles, and then apply again the Borel--Bott--Weil algorithm. However, note that by the appropriate Schur power of Equation \ref{eq:tautological_sequence_rank_5} we can resolve $\Sym^2\Uc$ by means of direct sums of wedge powers of $\Uc^\vee$, which are simpler objects. In fact, one has an exact sequence:
        \begin{equation}
        \label{eq:long_sequence_appendix}
            0\arw \Sym^2\Uc\otimes\Uc^\vee(-2)\arw \Sym^2\VV^{D_5}_{\omega_1}\otimes\Uc^\vee(-2)\arw \VV^{D_5}_{\omega_1}\otimes\Uc^\vee\otimes\Uc^\vee(-2)\arw\wedge^2\Uc^\vee\otimes\Uc^\vee(-2)\arw 0
        \end{equation}
        The two middle terms have no cohomology: in fact, the first one is a direct sum of $55$ copies of $\Uc^\vee(-2)\simeq \Ec_{\omega_1-2\omega_4}$, while the second one is a direct sum of $10$ copies of $\Uc^\vee\otimes\Uc^\vee(-2) = \Ec_{2\omega_1-2\omega_4}\oplus\Ec_{\omega_2-2\omega_4}$ and all the irreducible direct summands have no cohomology by Lemma \ref{lem:coh_general_D5}. On the other hand, the last term decomposes as:
        \begin{equation*}
            \wedge^2\Uc^\vee\otimes\Uc^\vee(-2) \simeq \Ec_{\omega_1+\omega_2-2\omega_4}\oplus\Ec_{\omega_3-2\omega_4}
        \end{equation*}
        and, while the first summand has no cohomology, the second contributes with $H^1(X, \Ec_{\omega_3-2\omega_4})\simeq\CC$ (again by Lemma \ref{lem:coh_general_D5}). In light of the sequence \ref{eq:long_sequence_appendix}, this completes the proof.
    \end{proof}
    \begin{lemma}\label{lem:ugly_ext}
        We have $\Ext^\bullet(\Sym^2\Uc^\vee, \wh T^\vee(-1)) = \CC[-2]$.
    \end{lemma}
    \begin{proof}
        This is equivalent to compute $H^\bullet(X, \Sym^2\Uc\otimes\wh T^\vee(-1))$. By the dual of the sequence \ref{eq:affine_tangent_bundle_as_a_kernel} we can easily resolve $\Sym^2\Uc\otimes\wh T^\vee(-1)$ as the kernel of the injection $\Sym^2\Uc\otimes\Uc^\vee(-2)\xhookrightarrow{\,\,\,\,\,}\VV^{D_5}_{\omega_4}\otimes \Sym^2\Uc(-1)$. The first bundle has cohomology $\CC[-3]$ by Lemma \ref{lem:ext_c3}, while $\Sym^2\Uc(-1)$ can be resolved by the sequence:
        \begin{equation*}
            0\arw\Sym^2\Uc(-1)\arw \Sym^2\VV^{D_5}_{\omega_1}\otimes\Oc(-1)\arw \VV^{D_5}_{\omega_1}\otimes\Uc^\vee(-1)\arw \wedge^2\Uc^\vee(-1)\arw 0
        \end{equation*}
        and all bundles have no cohomology by Lemma \ref{lem:coh_general_D5}. This completes the proof.
    \end{proof}
    \subsection{Cohomology computations on \texorpdfstring{$B_4/Q^4$}{something}}\label{subsec:bwb_B4}
    The notation of this paragraph has been established  in Section \ref{subsubsec:B4_description}.    \begin{lemma}\label{lem:cohomology_extension_U4_O}
        One has:
        \begin{equation*}
            \begin{split}
                H^\bullet(X, \wedge^3\Rc^\vee(-2)) & \simeq \CC[-1]\\
                H^\bullet(X, \Ec_{\nu_1+\nu_2-2\nu_4}) & = 0
            \end{split}
        \end{equation*}
    \end{lemma}
    \begin{proof}
        By Equation \ref{eq:bundles_to_weights_B4} we see that $\wedge^3\Rc^\vee(-2) = \Ec_{\nu_3-2\nu_4}$. We proceed with the same algorithm as before but for $B_4$. By adding $\rho$ we find:
        \begin{equation*}
            \dynkin[root radius = 3pt, edge length = 25pt,labels*={0, 0, 1, -2}]{B}{ooo*} \xrightarrow{\hspace{10pt} +\rho \hspace{10pt}}\dynkin[root radius = 3pt, edge length = 25pt,labels*={1, 1, 2, -1}]{B}{ooo*}  
        \end{equation*}
        By Theorem \ref{thm:borel_weil_bott}, we have cohomology if with a finite number of simple Weyl reflections we find a dominant weight: hence, we act with $S_{\beta_4}$ in order to get rid of the negative coefficient. We find:
        \begin{equation*}
            \dynkin[root radius = 3pt, edge length = 25pt,labels*={1, 1, 2, -1}]{B}{ooo*} \xrightarrow{\hspace{10pt} S_{\beta_4}\hspace{10pt}}\dynkin[root radius = 3pt, edge length = 25pt,labels*={1, 1, 1, 1}]{B}{ooo*}  
        \end{equation*}
        At this point, it is clear that if we subtract $\rho$ from the resulting weight, we get a dominant weight (the trivial one). Hence we conclude that the cohomology of $\Ec_{\nu_3-2\nu_4}$ is one dimensional in degree given by the number of simple Weyl reflections we used, which is one.\\
        \\
        Let us consider the second claim. Here, applying $\rho$ to the weight $\nu_1+\nu_2-2\nu_4$ we find:
        \begin{equation*}
            \dynkin[root radius = 3pt, edge length = 25pt,labels*={1, 1, 0, -2}]{B}{ooo*} \xrightarrow{\hspace{10pt} +\rho \hspace{10pt}}\dynkin[root radius = 3pt, edge length = 25pt,labels*={2, 2, 1, -1}]{B}{ooo*}  
        \end{equation*}
        Let us apply $S_{\beta_4}$:
        \begin{equation*}
            \dynkin[root radius = 3pt, edge length = 25pt,labels*={2, 2, 1, -1}]{B}{ooo*} \xrightarrow{\hspace{10pt} S_{\beta_4} \hspace{10pt}}\dynkin[root radius = 3pt, edge length = 25pt,labels*={2, 2, 0, 1}]{B}{ooo*}  
        \end{equation*}
        As above, we found a zero coefficient: Therefore, by Theorem \ref{thm:borel_weil_bott}, $\Ec_{\nu_1+\nu_2-2\nu_4}$ has no cohomology. 
    \end{proof}
    \begin{lemma}\label{lem:equivariant_extension_1}
        We have $\Ext_G^\bullet(\wedge^2\Rc^\vee, \Rc^\vee) = \CC[-1]$ 
    \end{lemma}
    \begin{proof}
        The proof follows immediately by the decomposition $\wedge^2\Rc\otimes\Rc^\vee \simeq \wedge^2\Rc^\vee\otimes\wedge^2\Rc^\vee(-2)\simeq\Ec_{\nu_1+\nu_2-2\nu_5}\oplus\Ec_{\nu_3-2\nu_4}$ together with lemma \ref{lem:cohomology_extension_U4_O}.
    \end{proof}
    \begin{corollary}\label{cor:cohomology_U4}
        One has $\Ext^\bullet(\Oc, \Rc) = \CC[-1]$.
    \end{corollary}
    \begin{proof}
        Since $\Ext^\bullet(\Oc, \Rc) \simeq H^\bullet(X, \Rc)$, the proof follows simply by the isomorphism $\Rc\simeq \wedge^3\Rc^\vee(-2)$ (Lemma \ref{lem:smirnov}) and by Lemma \ref{lem:cohomology_extension_U4_O}.
    \end{proof}
    \begin{lemma}\label{lem:vanishing_stability}
        The vanishing $H^0(X, \wedge^r\Rc(-c)) = 0$ holds for any $c\geq -1$ and $1\leq r\leq 3$.
    \end{lemma}
    \begin{proof}
        The proof follows the same identical reasoning of the one of Lemma \ref{lem:cohomology_extension_U4_O}, once we write $\wedge^r\Rc(-c) = \Ec_{\nu_r-c\nu_4}$.
    \end{proof}
    \begin{lemma}\label{lem:vanishing_sym2U4}
        One has $H^\bullet(X, \Sym^2\Rc) = 0$.
    \end{lemma}
    \begin{proof}
        A quick way to prove this claim is to resolve $\Sym^2\Rc$ in terms of $D_5$-homogeneous bundles and then apply some results of the previous subsection. By the appropriate Schur power of the sequence \ref{eq:sequence_of_tautologicals} we obtain:
        \begin{equation}\label{eq:sym2_mixed_tautologicals}
            0\arw \Sym^2\Rc\arw \Sym^2\Uc \arw \Uc \arw 0.
        \end{equation}
        The last bundle has no cohomology by Lemma \ref{lem:cohomology_U5}, while the middle term can be resolved by a Schur power of the sequence \ref{eq:tautological_sequence_rank_5}:
        \begin{equation}\label{sym2_tautological_U5}
            0\arw \Sym^2\Uc\arw \VV^{D_5}_{\omega_1}\otimes\Uc \arw \wedge^2 \VV^{D_5}_{\omega_1}\otimes\Oc \arw \wedge^2\Uc^\vee \arw 0.
        \end{equation}
        Note that the last arrow is an isomorphism at the level of cohomology (it is the evaluation map on sections and $\wedge^2\Uc^\vee$ is globally generated). Then, we conclude by Lemma \ref{lem:cohomology_U5} again.
    \end{proof}
    \begin{lemma}\label{lem:vanishing_mixed}
        We have $H^\bullet(X, \Sym^2\Rc\otimes\Uc^\vee)^G = \CC[-1]$.
    \end{lemma}
    \begin{proof}
        By the tensor power of the sequence \ref{eq:sym2_mixed_tautologicals} with $\Uc^\vee$ we find:
        \begin{equation}
        \label{eq:sym2_mixed_tautologicals_twisted}
            0\arw \Sym^2\Rc\otimes\Uc^\vee \arw \Sym^2\Uc\otimes \Uc^\vee \arw \Uc \otimes \Uc^\vee \arw 0.
        \end{equation}
        The second term is resolved again by the tensor product of $\Uc^\vee$ with a Schur power of the sequence \ref{sym2_tautological_U5}:
        \begin{equation*}
            0\arw \Sym^2\Uc\otimes\Uc^\vee \arw \Sym^2\VV^{D_5}_{\omega_1}\otimes\Uc^\vee \arw \VV^{D_5}_{\omega_1}\otimes\Uc^\vee\otimes\Uc^\vee  \arw \wedge^2\Uc^\vee \otimes\Uc^\vee \arw 0.
        \end{equation*}
        In this last sequence, the cohomology of all terms but the first can be computed by the Littlewood--Richardson formula together with Lemma \ref{lem:D5_dominant_weights}. We find $H^\bullet(X, \Sym^2\Uc\otimes\Uc^\vee)\simeq \VV^{D_5}_{\omega_1}[-1]$. On the other hand, the last term of the sequence \ref{eq:sym2_mixed_tautologicals_twisted} contributes with $\CC[-1]$ ($\Uc$ is an exceptional vector bundle in $\dbcoh(X)$). We conclude by taking the space of invariants.
    \end{proof}

\bibliographystyle{alpha}
\bibliography{biblio}

\end{document}